\documentclass[12pt,reqno,a4paper]{article}
\pdfoutput=1
\usepackage{url}
\usepackage{amsmath,amssymb,amsthm,mathtools}
\usepackage[lmargin=35mm,rmargin=35mm,bmargin=30mm,tmargin=30mm]{geometry}
\usepackage{bbm}
\usepackage[inline]{enumitem}
\usepackage{faktor}
\usepackage{xcolor}

\title{Counting partitions of $G_{n,1/2}$ with degree congruence conditions}
\author{Paul Balister\footnote{Mathematical Institute, University of Oxford, Oxford OX2\thinspace6GG, United Kingdom  \texttt{\{balister,powierski,scott,jane.tan\}@maths.ox.ac.uk}}~\footnote{Research supported by EPSRC grant EP/W015404/1.} \;
Emil Powierski\protect\footnotemark[1] \;
Alex Scott\protect\footnotemark[1]~\footnote{Research supported by EPSRC grant EP/V007327/1.} \;
Jane Tan\protect\footnotemark[1]}
\date{}

\newtheorem{theorem}{Theorem}

\newtheorem{lemma}[theorem]{Lemma}
\newtheorem{corollary}[theorem]{Corollary}

\newtheorem{problem}[theorem]{Problem}

\makeatletter
\newtheorem*{rep@theorem}{\rep@title}
\newcommand{\newreptheorem}[2]{%
\newenvironment{rep#1}[1]{%
 \def\rep@title{#2 \ref{##1}}%
 \begin{rep@theorem}}%
 {\end{rep@theorem}}}
\makeatother

\newtheorem*{conjecture*}{Conjecture}
\newreptheorem{conjecture}{Conjecture}
\theoremstyle{definition}                    

\theoremstyle{remark}   

\newtheorem*{remark*}{Remark}

\numberwithin{equation}{section}

\newcommand\F{\mathbb{F}}
\newcommand\Z{\mathbb{Z}}
\newcommand\N{\mathbb{N}}
\newcommand\E{\mathbb{E}}
\newcommand\Prb{\mathbb{P}}

\newcommand\eps{\varepsilon}
\newcommand\Po{\mathrm{Po}}
\newcommand\cc{\mathbf{c}}
\newcommand\ca{\mathbf{a}}

\newcommand\id{\mathbbm{1}}
\newcommand\A{\mathbf{A}}
\newcommand\B{B}
\newcommand\ones{\mathbf{1}}
\newcommand\zeroes{\mathbf{0}}
\newcommand\vv{\mathbf{v}}
\newcommand\ee{\mathbf{e}}
\newcommand\bb{\mathbf{b}}
\newcommand\zzeta{\boldsymbol{\zeta}}
\DeclareMathOperator\rank{rk}
\DeclareMathOperator\Hom{Hom}
\newcommand\Mc[1]{\def\temp{#1}M_{\rm{col}}^{\ifx\temp\empty\else(#1)\fi}}
\newcommand\Mcol[1]{\langle\Mc{#1}\rangle}

\begin{document}

\maketitle

\begin{abstract}
For $G=G_{n, 1/2}$, the Erd\H{o}s--Renyi random graph, let $X_n$ be the random variable representing the number of distinct partitions of $V(G)$ into sets $A_1, \ldots, A_q$ so that the degree of each vertex in $G[A_i]$ is divisible by $q$ for all $i\in[q]$. We prove that if $q\geq 3$ is odd then $X_n\xrightarrow{d}{\Po(1/q!)}$, and if $q \geq 4$ is even then $X_n\xrightarrow{d}{\Po(2^q/q!)}$. More generally, we show that the distribution is still asymptotically Poisson when we require all degrees in $G[A_i]$ to be congruent to $x_i$ modulo $q$ for each $i\in[q]$, where the residues $x_i$ may be chosen freely. For $q=2$, the distribution is not asymptotically Poisson, but it can be determined explicitly.\\

\end{abstract}

\section{Introduction}
A folklore result of Gallai (see~\cite{problems}, Exercise 5.17) states that every graph $G$ has a vertex partition into two parts $V_1$ and $V_2$ so that all degrees in the induced subgraphs $G[V_1]$ and $G[V_2]$ are even. An easy corollary of this is that there also exists a vertex partition into two parts for which the degrees in $G[V_1]$ are all odd whilst the degrees in $G[V_2]$ are all even. Another corollary in the same vein is a solution to a well-known riddle: given any graph with lights turned on at each vertex and buttons corresponding to each vertex that toggle the status (light on/off) of a vertex together with its neighbourhood, there is a sequence of button-pushes that turns all of the lights off.

Given that Gallai's theorem guarantees the existence of even/even and even/odd partitions into two parts, one line of research that has arisen is to investigate partitions where each part induces a subgraph with all degrees odd. By considering any graph with an odd number of vertices, it is clear that it is not always possible to find a partition into two parts satisfying this condition. However, improving on earlier work of Caro~\cite{Caro94} and Scott~\cite{Alex92}, Ferber and Krivelevich~\cite{FK20} recently proved that any graph without isolated vertices contains a linearly sized induced subgraph with all degrees odd. 

When we allow more than two parts, it is still trivially necessary that $G$ must have an even number of vertices in each part and hence in each component. The following result due to Scott~\cite{Alex01} states that this is in fact sufficient. We shall call a subgraph $H$ of a graph $G$ \emph{odd} if $d_H(v)$ is odd for all $v \in V(H)$, and similarly \emph{even} if all degrees are even.

\begin{theorem}[\cite{Alex01}]\label{thm:alex1}
The vertices of a graph $G$ can be partitioned into sets $A_1, \ldots, A_k$ for some $k$ such that $G[A_i]$ is an odd subgraph for all\/ $i$ if and only if every component of\/ $G$ has even order. 
\end{theorem}

More generally, we can consider congruence conditions modulo $q$ for $q\geq 2$. Caro, Krasikov and Roditty~\cite{CKR94a} asked whether there exists a number $k=k(q)$ such that every graph $G$ can be partitioned into $k$ vertex-disjoint classes $A_1,\dots, A_k$ in such a way that all degrees in all induced subgraphs $G[A_i]$ are divisible by~$q$. While this deterministic problem is very much open, more can be said in the context of random graphs. 

Henceforth let $G_{n,1/2}$ be the standard binomial random graph with vertex set $[n]$ and edge probability $1/2$. Scott~\cite{Alex01} formulated a random version of Caro, Krasikov and Roditty's problem in which we ask for partitions satisfying fixed degree residue conditions to exist for almost every $G_{n,1/2}$ (i.e.\ with probability tending to 1 as $n\to \infty$).

\begin{problem}[\cite{Alex01}]\label{prob:randomk} Let $q\geq 2$ and $0\leq x<q$ be integers. Does there exist a number $k=k(q,x)$ such that almost every $G_{n,1/2}$ \textup{(}with $n$ even for simplicity\textup{)} can be partitioned into $k$ vertex-disjoint classes $A_1,\dots, A_k$ in such a way that in all the induced subgraphs $G[A_i]$ all degrees are $x \bmod q$?
\end{problem}

It was shown in \cite{Alex01} that three parts suffice for $q=2$ and $x=1$. Ferber, Hardiman and Krivelevich~\cite{FHKtalk, FHKarxiv} recently solved Problem~\ref{prob:randomk} for general $q$, showing that $k(q,x) = q+1$ is sufficient for all $x$.

\begin{theorem}[Ferber, Hardiman and Krivelevich~\cite{FHKarxiv}]\label{thm:fhkmain}
For all\/ $q\in \N$ and\/ $0\leq x<q$, almost every $G_{n,1/2}$ has a vertex partition into $q+1$ sets $A_1,\dots,A_{q+1}$ such that in each induced subgraph $G[A_i]$ all degrees are $x \bmod q$.
\end{theorem}

It is not hard to show (by a first moment argument) that Theorem~\ref{thm:fhkmain} does not hold for partitions into $q-1$ parts. However, a natural question asked in~\cite{FHKtalk} is whether the theorem still holds with $q$ parts instead of $q+1$. In this paper, we provide a negative answer and moreover determine the asymptotic distribution of the number of `good' partitions into $q$ parts. The situation differs between the cases $q>2$ and $q=2$. Beginning with the former, for $G=G_{n, 1/2}$, let $X_n$ be the random variable representing the number of partitions of $V(G)$ into disjoint sets $A_1,\ldots, A_q$ so that the degree of each vertex in $G[V_i]$ is divisible by $q$ for all $i\in[q]$, i.e.
\[
 X_n=\Big|\big\{\{A_1,\dots, A_q\}: V(G)=\coprod_{i=1}^q A_i \text{ and }
 q\mid d_{G[A_i]}(v) \text{ for all }i\in[q],\, v\in A_i\big\}\Big|.
\]
Although we number the parts, we only consider partitions up to permutation of the parts. Our first result gives the distribution of $X_n$ as $n\to \infty$.

\begin{theorem}\label{thm:main}
If\/ $q\geq 3$ is odd, then  $X_n\xrightarrow{d}{\Po(1/q!)}$. If\/ $q \geq 4$ is even, then $X_n\xrightarrow{d}{\Po(2^q/q!)}$.
\end{theorem}

With some minor modifications to the proof of Theorem~\ref{thm:main}, we can obtain the following stronger statement which allows for greater flexibility in fixing congruence conditions. Specifically, for non-negative integers $a_0,\dots, a_{q-1}$ with $\sum_{x=0}^{q-1}a_x=q$ we now let $X_n=X_n^{(a_0,\dots , a_{q-1})}$ be the number of $q$-tuples $(\{A_{0,1},\dots,A_{0,a_0}\}, \{A_{1,1},\dots,A_{1,a_1}\},\dots)$, where each entry is an (unordered, possibly empty) set of parts and each part satisfies the degree condition given by its first index, i.e.
\begin{align*}
X_n=\Big|\big\{&(\{A_{0,1},\dots,A_{0,a_0}\},\{A_{1,1},\dots,A_{1,a_1}\},\dots \{A_{q-1,1},\dots,A_{q-1,a_{q-1}}\}): \\ 
&V(G)=\coprod_{(x,y)\in T} A_{x,y}\text{ and }
 d_{G[A_{x,y}]}(v)\equiv x \text{ for all }(x,y)\in T, \, v\in A_{x,y}\big\}\Big|,
\end{align*}
where $T=\{(x,y): x=0,\dots,q-1, y=1,\dots, a_x\}$ and the congruence is mod $q$. This definition ensures that (as before) two partitions that can be transformed into each other by reordering parts with the same degree condition are considered the same and only counted once. Setting $a_0=q$ and $a_i=0$ for $i \neq 0$ we would get the random variable we considered in Theorem~\ref{thm:main}.
\begin{theorem}\label{thm:main2}
Let $G=G_{n,1/2}$ and let $X_n=X_n^{(a_0,\dots , a_{q-1})}$ be as defined above for any non-negative integers $a_0,\dots,a_{q-1}$ with $\sum_{x=0}^{q-1}a_x=q$. 
For $q$ even, write $c=\sum_{x=0}^{q/2-1} a_{2x}$ for the number of parts where the degree condition is even.
\begin{enumerate}[label={(\arabic*)}]
\itemsep=0mm
 \item If\/ $q\ge 3$ is odd, then  $X_n\xrightarrow{d}{\Po(1/\prod a_x!)}$.
 \item If\/ $q\ge 4$ is even and $c>0$, then $X_n\xrightarrow{d}{\Po(2^c/\prod a_x!)}$.
 \item If\/ $q\ge 4$ is even and $c=0$, then $X_n\xrightarrow{d}{\Po(2/\prod a_x!)}$ as $n$ runs over
 \emph{even} integers.
\end{enumerate}
\end{theorem}
Note that if $q$ is even and $c=0$ then all degrees in each $G[A_{x,y}]$ must be odd, so each $|A_{x,y}|$ is even
and hence $n=\sum|A_{x,y}|$ must also be even.

For a cleaner presentation, Theorem~\ref{thm:main} and  Theorem~\ref{thm:main2} feature the random graph with edge probability $p=\frac{1}{2}$, although our proofs actually extend to any $p=p(n)$ with $C \frac{\log n}{n}\le p\le 1-C \frac{\log n}{n}$
for some large constant~$C$. The proofs of these theorems will be presented in Sections~\ref{sec:mainproof}--\ref{sec:LA}.

The case where $q=2$ is really exceptional, which is especially intriguing given that the following question was one of the starting points in the random setting.

\begin{problem}[\cite{Alex01}]\label{prob:oddodd}
For $n$ even, what is the probability that $G_{n,1/2}$ can be partitioned into two sets, each inducing odd subgraphs?
\end{problem}
It is known~\cite{Alex01} that the answer is at least $1/2+o(1)$. We show that the answer to Problem~\ref{prob:oddodd} is $2/3+o(1)$. In fact, we provide the full distributions of $X^{(2,0)}_n$, $X^{(1,1)}_n$ and $X^{(0,2)}_n$, recalling that these are the number of bipartitions of $G_{n,1/2}$ inducing even/even, even/odd and odd/odd partitions respectively. The exact distributions are given in Section~\ref{sec:final}, and lead to the following asymptotic distributions.

\begin{theorem}\label{thm:q=2asympt}
Let\/ $G=G_{n,1/2}$. Then $X_n^{(2,0)} \xrightarrow{d} X$ and\/ $X^{(1,1)}_n\xrightarrow{d}X$ where 
\[
\Prb(X=2^k)=c \prod_{i=1}^{k} (2^i-1)^{-1}
\]
for $k \in \N\cup\{0\}$ with constant $c=\prod_{i=0}^{\infty}(1-2^{-2i-1})=(\sum_{j=1}^{\infty}\prod_{i=1}^{j} (2^i-1)^{-1})^{-1}$, and\/ $\Prb(X=x)=0$ if\/ $x \neq 2^k$ for any $k\in \N\cup\{0\}$. Furthermore, $X^{(0,2)}_n \xrightarrow{d} Z$ where $n$ only runs over even integers and  
\[
\Prb(Z=2^k)= c 2^{-k} \prod_{i=1}^{k} (2^i-1)^{-1}
\]
for $k\in \N\cup\{0\}$, $\Prb(Z=0)=\frac{1}{3}$ and\/ $\Prb(Z=x)=0$ if\/ $x\neq 0$ and $x \neq 2^k$ for any $k\in \N\cup\{0\}$.
\end{theorem}

\section{Proof of Theorem~\ref{thm:main}}\label{sec:mainproof}

For some $n,q\in \N$, consider a graph $G = ([n], E)$ and a partition $A=\{A_i\}_{i\in[q]}$ of $[n]$ into $q$ parts.
Define the \emph{subgraph induced by $A$} to be the disjoint union of the subgraphs induced by each part, that is, the spanning subgraph 
\[G[A]:= G[A_1]\coprod\dots\coprod G[A_q].\]
Equivalently, $G[A]$ can be obtained from $G$ by removing all edges for which the endvertices lie in different parts of $A$. We shall call a partition $A$ \emph{good} if the degree of every vertex in $G[A]$ is divisible by $q$. 

With $G = G_{n,1/2}$, let $X = X_n$ be the random variable representing the number of good partitions. In order to show that $X$ is asymptotically Poisson distributed, we use the method of moments. Let $\E((X)_k):=\E(X(X-1)\dots(X-k+1))$ denote the falling factorial expectation of $X$.

\begin{theorem}[see e.g.\ \cite{BBrg}, Theorem 1.22]\label{thm:poisson}
Let $(X_n)_{n\in{\N}}$ be a sequence of non-negative integer-valued random variables, and\/ $\lambda\geq 0$. If 
\[
\lim_{n \to \infty} \E((X_n)_k) = \lambda^k 
\]
for $k=0,1,\ldots$, then $X_n$ converges in distribution to $\Po(\lambda)$ as $n\to\infty$.
\end{theorem}

In view of Theorem~\ref{thm:poisson}, the proof of Theorem~\ref{thm:main} entails counting the expected number of ordered $k$-tuples
$\A = (A^{(1)},\dots,A^{(k)})$ of \emph{distinct\/} partitions $A^{(j)}$ such that each $A^{(j)}$ is good. We use discrete Fourier analysis to write a deterministic expression for the expectation, which we then compute via combinatorial and algebraic means. It turns out that this count mostly comprises $k$-tuples which intersect `generically', or are `independent' in some sense.
To describe this situation, we introduce some notation that will be used throughout this section as well as Section~\ref{sec:proof2}.

For a particular $\A$ let the \emph{coordinates} of a vertex $v\in[n]$ be the $k$-tuple $\cc(v)=(c_1,\dots,c_k)$ such that $v\in (A^{(j)})_{c_j}$ for all $j$. Then for vertices $u,v\in [n]$, let $I_{u,v}:= \{j\in [k]: c(u)_j=c(v)_j\}$. That is, $I_{u,v}$ corresponds to the set of partitions in our $k$-tuple in which $u$ and $v$ lie in the same part. Given $\cc,\cc'\in [q]^k$, we will similarly let $I_{\cc,\cc'}=\{j\in[k]:c_j=c'_j\}$. It is quite possible for many vertices to share the same coordinates. Indeed, given $\cc\in[q]^k$ we define the \emph{box} $V_\cc=\{v\in[n]:\cc(v)=\cc\}$ to be the set of vertices with coordinates~$\cc$. Equivalently,
\[V_\cc = \bigcap_{j\in [k]} (A^{(j)})_{c_j}.\] 
We will show that the $k$-tuples of partitions for which all of the associated boxes are reasonably large (and hence the partitions will be essentially independent) contribute $q!^{-k}+o(1)$ to the expectation $\E((X)_k)$ when $q$ is odd and $2^q q!^{-k}+o(1)$ when $q$ is even, whilst the remaining configurations contribute $o(1)$ as $n\to\infty$. Since we are just concerned with asymptotics, all statements made throughout should be interpreted with the implicit assumption that $n$ is sufficiently large.

To start the proof, let $q \geq 3$ and fix a $k$-tuple of partitions of $[n]$, say $\A = (A^{(1)},A^{(2)},\dots,A^{(k)})$. Our goal is to determine the probability over choice of $G_{n,1/2}$ that all partitions $A^{(j)}$ in this tuple are good. Let $\zeta_q=e^{2\pi i/q}$ ($i=\sqrt{-1}$ in this instance only) and $\mu_q=\{1,\zeta_q,\dots,\zeta_q^{q-1}\}$ be the $q$th roots of unity. We denote by $\zzeta=(\zeta_{v,j})$ an assignment of roots of unity to vertices for each partition, consisting of a root $\zeta_{v,j}\in \mu_q$ corresponding to each vertex $v$ and $j\in[k]$. Given some $\zzeta$, we denote the $k$-tuple of roots for a fixed vertex $v$ by $\zzeta_v = (\zeta_{v,1}, \ldots, \zeta_{v,k})$. Let $R$ be the set of all possible assignments $\zzeta$, so that $|R| = q^{kn}$. 
For each $d\in \Z$ we have
\begin{equation}\label{eq:zetacases}
    \frac{1}{q}\sum_{\zeta \in \mu_q} \zeta^d = \begin{cases}
    1, & \text{if } q\mid d; \\
    0, & \text{otherwise.} \\
  \end{cases}
\end{equation}
Thinking of $d$ as the degree of a vertex $v$ in the spanning subgraph $G_j:= G[A^{(j)}]$ induced by the partition $A^{(j)}$, the assigned roots will be used to determine whether $v$ satisfies the required degree condition. Writing $\id\{\ldots\}$ for the indicator function of an event and $\E$ for the expectation over choices of the random graph $G_{n,1/2}$, it follows from \eqref{eq:zetacases} that
\begin{align*}
\Prb(\text{all $A^{(j)}$ are good})
&=\E \;\mathbbm{1}\{ \text{$d_{G_j}(v) \equiv 0 \bmod q$ for all $v$ and $j$}\}\\
&=\E \prod_{j \in [k]} \prod_{v \in [n]}\frac{1}{q}\sum_{\zeta \in \mu_q}\zeta^{d_{G_j}(v)}\\
&= \E \frac{1}{q^{kn}} \sum_{\zzeta \in R} \prod_{j \in [k]} \prod_{v \in [n]} \zeta_{v,j}^{d_{G_j}(v)}\\
&= \E \frac{1}{q^{kn}} \sum_{\zzeta\in R} \prod_{j \in [k]} \prod_{vw \in E(G_j)} \zeta_{v,j} \zeta_{w,j}\\
&= \frac{1}{q^{kn}} \sum_{\zzeta\in R} \E \prod_{vw \in E(G)} \prod_{j \in I_{v,w}} \zeta_{v,j} \zeta_{w,j}\\
&= \frac{1}{q^{kn}} \sum_{\zzeta\in R\,} \prod_{\,\{v,w\}} \frac{1}{2}\Big(1+ \prod_{j \in I_{v,w}} \zeta_{v,j} \zeta_{w,j}\Big)
\end{align*}
where the last equality holds because $G=G_{n,1/2}$, and the outer product is taken over all 2-element subsets of $[n]$. We can then write
\begin{equation}\label{eq:mainexp}
    \E((X)_k)=\frac{1}{q^{kn}} \sum_{\A} \sum_{\zzeta\in R\,} \prod_{\,\{v,w\}}\frac{1}{2}\Big(1+\prod_{j\in I_{v,w}}\zeta_{v,j}\zeta_{w,j}\Big)  
\end{equation}
with the first sum taken over all $k$-tuples of (not necessarily good) distinct partitions. The expression in (\ref{eq:mainexp}) is deterministic in the sense that it no longer involves the random graph, and we are instead left to work with \emph{configurations}, which are choices $(\A, \zzeta)$ of a $k$-tuple of partitions and assignment of roots. To dismiss the possibility that the contributions from different configurations, which in general are complex numbers, may cancel each other out, we will work with the expression
\begin{equation}\label{eq:mainexp_mod}
    \frac{1}{q^{kn}} \sum_{\A} \sum_{\zzeta\in R\,} \prod_{\,\{v,w\}} \Big| \frac{1}{2}\Big(1+\prod_{j\in I_{v,w}}\zeta_{v,j}\zeta_{w,j} \Big) \Big|,
\end{equation}
where we take the modulus of the individual contributions.
For the computation, we will group together configurations that share certain characteristics and bound and compare the contributions of those groups to \eqref{eq:mainexp} and \eqref{eq:mainexp_mod}.
For a given configuration, observe that the contribution is small unless almost all pairs $v,w$ satisfy $\prod_{j\in I_{v,w}}\zeta_{v,j}\zeta_{w,j}=1$. Specifically, say that distinct vertices $v$ and $w$ are in \emph{conflict\/} if $\prod_{j\in I_{v,w}}\zeta_{v,j}\zeta_{w,j}\ne1$. 
For every conflicted pair,
\begin{equation}\label{eq:conflictbound}
 \Big|\frac{1}{2}\Big(1+\prod_{j\in I_{v,w}}\zeta_{v,j}\zeta_{w,j}\Big)\Big|\le \frac{1}{2}\big|1+\zeta_q\big|=\cos(\pi/q)\le e^{-1/q^2}.
\end{equation}
We will in fact see that the total contribution of all configurations with conflicts is $o(1)$. This will show that \eqref{eq:mainexp} and \eqref{eq:mainexp_mod} converge to the same limit, so that we may consider contributions to \eqref{eq:mainexp_mod} rather than directly to $\E((X)_k)$.

Fix $K>kq^2\log q$. Since there are at most $q^{kn}$ choices of $k$-tuples $\A$ and the same number of possible assignments $\zzeta$, the total contribution to \eqref{eq:mainexp_mod} from configurations with more than $Kn$ conflicted pairs is at most
\begin{equation}\label{eq:tooconflicted}
q^{-kn}\cdot q^{2kn} \cdot \big(e^{-1/q^2}\big)^{Kn} = e^{-(K-kq^2\log q)n/q^2} = o(1)
\end{equation}
as $n\to\infty$. 

Now fix $C>2q^2$, and call a vertex \emph{bad\/} if it is involved in more than $C\log n$ conflicted pairs. A vertex that is not bad is \emph{good}.
From (\ref{eq:tooconflicted}), we may assume that there are at most $Kn$ conflicted pairs and hence, allowing for the possibility that bad vertices may be in conflict with each other, at most $2Kn/(C\log n)=o(n)$ bad vertices.

Since there are $q^k$ boxes and $q^k$ choices of $\zzeta_v$ for each vertex $v$, there
must be a particular box $V^\star = V_\cc$ and vertex $v^\star\in V^\star$ such that $\zzeta_v = \zzeta_{v^\star}$ for at least $q^{-2k}n$ vertices $v \in V^\star$. Such a vertex (in $V^\star$ and assigned the same $k$-tuple of roots as $v^\star$) will be called a \emph{most common} vertex. 

\begin{lemma}\label{lemma:mcprops}
Let $(\A,\zzeta)$ be a configuration with $o(n)$ bad vertices. Then all most common vertices are good. Moreover, each most common vertex is not conflicted with any other good vertices, including other most common vertices.
\end{lemma}
\begin{proof}
It is enough to observe that since all most common vertices $v$ have the same coordinates and $\zzeta_v$, they must all be in conflict with the same set of vertices. As there are only $o(n)$ bad vertices, some most common vertices are good, meaning they are in conflict with at most $C\log n$ vertices. Similarly, for the second statement, if any good vertex is in conflict with a most common vertex then it must be in conflict with all most common vertices, but this exceeds the allowed number $C\log n$ of conflicts for a good vertex.
\end{proof}

We will show that bad vertices may be replaced by most common vertices at a small cost in contribution to \eqref{eq:mainexp_mod}. This will allow us to assume that there are no bad vertices in the remainder of the argument.
\begin{lemma}\label{lemma:removebad}
The total contribution to \eqref{eq:mainexp_mod} from configurations with bad vertices is $o(1)$ times the contribution from configurations where all vertices are good. 
\end{lemma}
\begin{proof}
For a $k$-tuple of partitions $\A$ and assignment $\zzeta$, let $t$ be the number of bad vertices. We modify the configuration by replacing all the bad vertices
by duplicates of the most common vertices. That is
remove all $t$ of them and add $t$ vertices to the box $V^\star$, each assigned~$\boldsymbol{\zeta^\star}$. Note that this produces a configuration in which all vertices are good. Indeed, removing vertices cannot increase the number of conflicts, and the most common vertices are not conflicted with any good vertex by Lemma~\ref{lemma:mcprops}.

The preceding construction sends at most $\binom{n}{t}(q^k)^t(q^k)^t\le (nq^{2k})^t$ configurations that have $t$ bad vertices to a single configuration without bad vertices. To see this, reversing the process allows at most $\binom{n}{t}$ choices for the bad vertices, $q^k$ choices each for which box they were in and $q^k$ choices each for their
original $\zeta$ values. On the other hand, we have removed at least $t(C\log n)/2$ conflicts as after this transformation the `new' most common vertices do not participate in any conflicts by Lemma~\ref{lemma:mcprops}. Using the bound from (\ref{eq:conflictbound}) and the choice of $C>2q^2$, the sum of contributions to \eqref{eq:mainexp_mod} from all configurations with bad vertices that are transformed into a particular configuration with all vertices good divided by the contribution from that particular configuration is a factor of at most
\[
 \sum_{t>0}(nq^{2k})^te^{-(t(C\log n)/2)/q^2}
 = \sum_{t>0}\big(n^{1-C/2q^2}q^{2k}\big)^t
 =o(1)
\]
as $n\to\infty$. This is true for all configurations without bad vertices, so the lemma follows.
\end{proof}

Henceforth, we assume that there are no bad vertices.
This has some useful consequences for boxes $V_\cc$ with $|V_\cc| > 2C\log n+2$. We call such boxes \emph{large}, and the remaining non-empty boxes \emph{small}. Given any vertex $v\in [n]$ and $I\subseteq [k]$ define $\zeta_{v,I}:=\prod_{j\in I}\zeta_{v,j}$.

\begin{lemma}\label{lemma:nolargeconflict}
Let $(\A,\zzeta)$ be a configuration without bad vertices. Take any pair of boxes $V_\cc$ and\/ $V_{\cc'}$
\textup{(}possibly the same box\textup{)},
with the former being large, and let $I=I_{\cc,\cc'}$. Then we have $\zeta_{u,I}$ = $\zeta_{v,I}$ for all\/ vertices $u,v\in V_{\cc'}$.
If\/ $V_\cc$ and\/ $V_{\cc'}$ are both large boxes, then $\zeta_{v,I} = \zeta_{w, I}^{-1}$ for all\/ $v\in V_{\cc'}$ and\/ $w\in V_{\cc}$. In particular, there are no conflicted pairs within or between large boxes.
\end{lemma}
\begin{proof}
If $v\in V_{\cc'}$ and $w\in V_{\cc}$ then $v$ and $w$ are in conflict iff $\zeta_{v,I}\ne\zeta_{w,I}^{-1}$, where $I=I_{\cc,\cc'}$.
Since the $v\in V_{\cc'}$ are good we deduce that the $\zeta_{v,I}$ must all be equal to more than $|V_{\cc}|-1-C\log n>|V_{\cc}|/2$
values of $\zeta_{w,I}^{-1}$.
(The $-1$ is due to the fact that $v$ lies in $V_{\cc}$ when
$\cc=\cc'$.)
This proves the first statement. Assuming that $V_\cc$ and $V_{\cc'}$ are both large, applying the first result twice shows that there is a common $\zeta_{v,I}$ for all $v\in V_{\cc'}$ and also a common $\zeta_{w,I}$ for all $w\in V_{\cc}$, and moreover that $\zeta_{v,I} = \zeta_{w, I}^{-1}$ must be true for all $v\in V_{\cc'}$, $w\in V_{\cc}$. The final statement follows immediately, noting that we may take $\cc = \cc'$.
\end{proof}

In light of the first statement in Lemma~\ref{lemma:nolargeconflict}, one can define a common value $\zeta_{\cc',I}$ equal to $\zeta_{v,I}$ for all $v\in V_{\cc'}$ provided there is some large $V_{\cc}$ with $I=I_{\cc,\cc'}$.

We now begin to evaluate (\ref{eq:mainexp_mod}) by grouping together $k$-tuples in the first sum depending on how their component partitions intersect. For a $k$-tuple of partitions $\A$, define $L = L_\A := \{\cc \in [q]^k: V_\cc\text{ is large}\}$. 
Suppose that $L = [q]^k$, meaning all boxes are large. In this case we note that for each $\cc$, any $j\in [k]$ and any $v\in V_\cc$, we can find vertices $u, w$ that satisfy 
\[I_{v,u}=I_{v,w}=I_{u,w}=\{j\}.\]
To see this, take for instance $u\in V_{\cc'}$ and $w\in V_{\cc''}$ where $c''_j=c'_j=c_j$, but $c''_i,c'_i,c_i$ are distinct for all $i\ne j$. Such coordinates exist since $q\geq 3$. 
Now by Lemma~\ref{lemma:nolargeconflict} there are no conflicted vertices between large boxes, so $I_{v,u} = \{j\}$ implies that $\zeta_{v,j}=\zeta_{u,j}^{-1}$ and similarly for the other two pairs of vertices. This gives $\zeta_{v,j}=\zeta_{u,j}^{-1}=\zeta_{w,j}=\zeta_{v,j}^{-1}$, so $\zeta_{v,j}\in \{\pm1\}$. Moreover, for any $u$ with $\cc(u)_j=\cc(v)_j$ we can find a $w$ with $I_{v,w}=I_{w,u}=\{j\}$ by picking values to ensure that each $\cc(w)_i\ne \cc(u)_i,\cc(v)_i$, $i\ne j$, and $\cc(w)_j=\cc(u)_j=\cc(v)_j$. Then
$\zeta_{v,j}=\zeta_{w,j}^{-1}=\zeta_{u,j}$. Hence $\zeta_{v,j}$ depends only on the value of $\cc(v)_j$.
Thus, we can write $\zeta_{v,j}=\zeta_{\cc(v)_j,j}$ for some choice of $\zeta_{i,j}\in\{\pm1\}$ with $i\in[q]$ and $j\in[k]$.
Conversely any such choice gives rise to no conflicts. Indeed, for $I=I_{v,w}$ we have
\[
 \zeta_{v,I}\zeta_{w,I}=\prod_{j\in I}\zeta_{\cc(v)_j,j}\zeta_{\cc(w)_j,j}=\prod_{j\in I}\zeta_{\cc(v)_j,j}^2=1.
\]
Thus there are precisely $2^{kq}$ choices of $\zeta$ values giving rise to no bad vertices when $q$ is even, and only one (all $\zeta_{v,j}=1$) when $q$ is odd as then $-1\notin\mu_q$.

Now there are $q^n$ \emph{ordered\/} partitions
(allowing empty parts). An unordered partition without empty parts corresponds to exactly $q!$ ordered partitions
and so, as there are only $O((q-1)^n)$ partitions
with empty parts, we have $(1+o(1))q^n/q!$ unordered
partitions, whether or not we allow empty parts.

Also, it is easy to see that only $o(q^{kn})$ $k$-tuples of partitions have $L \neq [q]^k$. Indeed, we may choose a $k$-tuple of \emph{not necessarily distinct} ordered partitions uniformly at random by independently including each vertex in any box with probability $q^{-k}$. The probability that a fixed box is not large is $o(1)$. The assertion then follows by taking a union bound over all boxes and noting that the number of not necessarily distinct ordered partitions gives an upper bound on the number of tuples of distinct unordered partitions.

The number of $k$-tuples with $L = [q]^k$ is then $(q^n/q!)_k (1+o(1))$ and hence the total contribution to \eqref{eq:mainexp_mod} from these $k$-tuples is
\begin{equation}\label{eq:fullcontribodd}
q^{-kn} \cdot (q^n/q!)_k (1+o(1)) = (1/q!)^k + o(1)
\end{equation}
for $q$ odd and
\begin{equation}\label{eq:fullcontribeven}
q^{-kn}\cdot 2^{kq} \cdot(q^n/q!)_{k}(1+o(1))=(2^q/q!)^k+o(1)
\end{equation}
when $q$ is even.

A special family of subsets of $[q]^k$ are the \emph{combinatorial subspaces}, which are defined to be sets of the form
\begin{equation}\label{e:comb}
 \big\{(\phi_1(x_{i_1}),\dots,\phi_k(x_{i_k}))\in[q]^k: x_1,\dots,x_r\in[q]\big\}
\end{equation}
where $\phi_1,\dots,\phi_k$ are permutations of $[q]$ and $i_1,\dots,i_k\in[r]$.
In other words, up to permutations, each coordinate follows one of the variables~$x_i$, but different coordinates may follow the same variable. Equivalently, it is a non-empty set that can be expressed as the intersection of some number of constraints of the form $c_i=\phi_{ij}(c_j)$ where the $\phi_{ij}$ are permutations of~$[q]$.

Combinatorial subspaces capture the situation where two partitions in the tuple are the same, accounting for relabelling of parts. That is, if $L \neq [q]^k$ but forms a combinatorial subspace, and there are no small boxes, then two of the partitions in $\A$ must be the same. Indeed, in the notation of \eqref{e:comb}, under these assumptions there must exist $j,j' \in [k]$ such that $i_j=i_{j'}$. Then for $\phi=\phi_{j'} \circ \phi_j^{-1}$ we have $A^{(j)}_i=A^{(j')}_{\phi(i)}$ for all $i\in [q]$. To see this, letting $v\in A^{(j)}_i$ we get $v\in V_c$ with $c\in L$ as there are no small boxes. It follows that $c_{j'}=\phi(c_j)=\phi(i)$, i.e.\ $v \in A^{(j')}_{\phi(i)}$ as desired.

However, such $k$-tuples where two partitions are the same do not occur in \eqref{eq:mainexp}. The $k$-tuples yet to be considered therefore fall into one of two classes: those for which $L$ is not a combinatorial subspace and there are no small boxes, and those for which there are small boxes. It suffices to prove that the contribution to \eqref{eq:mainexp_mod} is $o(1)$ in both cases. Then, configurations with conflicts contribute $o(1)$ in total to \eqref{eq:mainexp_mod}, and \eqref{eq:mainexp} and \eqref{eq:mainexp_mod} converge to the same limit. This, together with \eqref{eq:fullcontribodd} and \eqref{eq:fullcontribeven}, would allow us to conclude that $\E((X)_k) \to (1/q!)^k$ for $q\geq 3$ odd and $\E((X)_k) \to(2^q/q!)^k$ for $q\geq 4$ even as $n \to \infty$, whence Theorem~\ref{thm:poisson} completes the proof.

To handle the remaining cases, we introduce some notation. Given a set $\B \subseteq [q]^{k}$ (which we will always choose to be either $L$ or the set $D:=\{\cc \in [q]^{k}:V_\cc \neq \emptyset\}$ of all non-empty boxes), to each large box~$V_\cc$, $\cc\in L \subseteq B$, we associate the matrix $M^{(\cc,\B)}=\big(M^{(\cc,\B)}_{j,\cc'}\big)_{j\in[k],\cc'\in \B}$ where $M^{(\cc,\B)}_{j,\cc'}=\id\{c_j=c'_j\}$. In other words $M^{(\cc,\B)}$ is a $k \times |B|$ matrix and each column is a 0-1 vector of length $k$ corresponding to some $\cc'\in B$ such that there is a 1 in row $j$ if $\cc$ and $\cc'$ agree in the $j$th component, and 0 otherwise.

The columns of $M^{(\cc,\B)}$ may be viewed as elements of $(\Z/q\Z)^k$. Let $\Mcol{\cc,\B}$ be the subgroup of $(\Z/q\Z)^k$ generated by the columns of $M^{(\cc,B)}$ and define
\[
 N_{\cc,\B}=|(\Z/q\Z)^k/\Mcol{\cc,\B}|
\]
to be the size of the quotient group. This quantity is useful due to the following two lemmas which tie combinatorial subspaces to the present algebraic setup. 

\begin{lemma}\label{lem:nc}
If there is at least one solution $\boldsymbol{a}\in (\Z/q\Z)^k$ to the congruence
\begin{equation}
 \boldsymbol{a}M^{(\cc, B)}\equiv \boldsymbol{b}\bmod q
\end{equation}
for fixed\/ $\boldsymbol{b}\in (\Z/q\Z)^{|B|}$, then the total number of solutions is given by $N_{\cc,B}$.
\end{lemma}

\begin{lemma}\label{lem:ineq_composite}
 Let $q\geq 3$. For all\/ $L\subseteq [q]^k$ we have
 \begin{equation}\label{e:qsum}
  \sum_{\cc\in L} N_{\cc,L}\le q^k.
 \end{equation}
Equality holds if and only if\/ $L$ is a combinatorial subspace.
\end{lemma}

The proofs of both Lemma~\ref{lem:nc} and Lemma~\ref{lem:ineq_composite} are deferred to Section~\ref{sec:LA}. We use them in the following lemma, in which we bound the contributions of two particular configuration types. These will allow us to complete the main computation; of the two classes of remaining configurations identified earlier, the first where $L$ is not a combinatorial subspace and there are no small boxes will be covered directly by (i) below, whilst the second in which there are small boxes can be reduced to the special case of (ii). Crucially, the configurations considered are instances for which $\sum_{\cc\in L} N_{\cc,D} < q^k$.

\begin{lemma}\label{lemma:zetachoice}
\begin{enumerate}[label=(\roman*)]
\item\label{case_noplane}
 The total contribution to \eqref{eq:mainexp_mod} from configurations without conflicts, without small boxes, and where $L$ is not a combinatorial subspace is at most $2^{q^k}q^{q^{2k}}(1-q^{-k})^n$.
\item\label{case_smallboxes}
 The total contribution to \eqref{eq:mainexp_mod} from configurations without conflicts and with exactly one small box which has size $1$ is at most $2^{q^k}q^{q^{2k}}n(1-q^{-k})^{n-1}$.
\end{enumerate}
\end{lemma}

\begin{proof}
Given a particular $\A$ we determine the contribution from all configurations $(\A,\zzeta)$ of the types described above. Since these configurations do not admit conflicts involving large boxes, for all large boxes $V_\cc$ and non-empty boxes $V_{\cc'}$ we can define a common value $\zeta_{\cc,I_{\cc,\cc'}}$ equal to $\zeta_{v,I_{\cc,\cc'}}$ for all $v\in V_{\cc}$.
Now also fix a choice of these $\zeta_{\cc,I_{\cc,\cc'}}$ and let $\zeta_{\cc,I_{\cc,\cc'}}=\zeta_q^{b_{\cc,\cc'}}$ for all $\cc\in L, \cc' \in D$. We determine the number of assignments $\zzeta$ that comply with these values of $\zeta_{\cc,I_{\cc,\cc'}}$. Fixing a large box $V_\cc$ and writing $\zeta_{v,j}=\zeta_q^{a_{v,j}}$ for $v \in V_\cc$ leads to the constraints
\begin{equation*}
\sum_{j\in I_{\cc,\cc'}}a_{v,j}\equiv b_{\cc,\cc'}\bmod q,
\end{equation*}
one for every $\cc'\in D$. These can also be formulated as
\begin{equation*}
 \boldsymbol{a}M^{(\cc,D)}\equiv \boldsymbol{b}\bmod q,
\end{equation*}
where $\boldsymbol{a}=(a_{v,j})_{j\in [k]}\in (\Z/q\Z)^k$ and $\boldsymbol{b}=(b_{\cc,\cc'})_{\cc'\in [D]}\in(\Z/q\Z)^{|D|}$. Then, with $\boldsymbol{b}$ fixed,  Lemma~\ref{lem:nc} tells us that the number of solutions $\boldsymbol{a}$ to this system is at most $N_{\cc,D}$. Thus, there are at most $N_{\cc,D}^{|V_\cc|}$ choices of $\zzeta$ for vertices in $V_\cc$. 
By repeating the argument for each box and noting that there are at most $q^{q^{2k}}$ choices of $(b_{\cc,\cc'})_{\cc,\cc'}$, we find that for our fixed $\A$ the number of assignments $\zzeta$ for which there are no conflicts involving vertices in large boxes is bounded by
\[
 q^{q^{2k}}\prod_{\cc \in L} N_{\cc,D}^{|V_\cc|}.
\]
It follows that the configurations for a fixed $L$ (and coordinates $\cc^s$ for the small box in case \ref{case_smallboxes}) and fixed box sizes $|V_\cc|$ for $\cc \in L$ contribute at most
\[
 \frac{1}{q^{kn}}\frac{n!}{\prod_{\cc \in L} |V_\cc|!}\cdot q^{q^{2k}}\prod_{\cc \in L} N_{\cc,D}^{|V_\cc|}
\]
to the expectation in~\eqref{eq:mainexp_mod}. Here, the second factor is the multinomial coefficient representing the choices of $\A$ that lead to the box sizes that were fixed before. Note that this is also valid for case \ref{case_smallboxes} since $|V_{c^s}|=1$.

In case \ref{case_noplane}, since $L$ is not a combinatorial subspace we obtain from Lemma~\ref{lem:ineq_composite} that $\sum_{\cc\in D}N_{\cc,D}=\sum_{\cc\in L}N_{\cc,L} < q^k$.  Taking a sum over possible box sizes and applying the multinomial theorem, we see that the contribution to \eqref{eq:mainexp_mod} from configurations of type \ref{case_noplane} with just a fixed $L$ is at most
\begin{align}\label{eq:nochcontr1}
 \sum_{(|V_\cc|:\cc\in L)}\frac{n!}{\prod_{\cc \in L} |V_\cc|!}\cdot q^{-kn}q^{q^{2k}}\prod_{\cc\in L} N_{\cc,D}^{|V_\cc|}
 &= q^{q^{2k}}q^{-kn}\Big(\sum_{c\in L}N_{\cc,D}\Big)^n \nonumber\\
 &\le q^{q^{2k}}q^{-kn}(q^k-1)^n \nonumber\\
 &= q^{q^{2k}}(1-q^{-k})^n.
\end{align}
This establishes statement \ref{case_noplane} as there are at most $2^{q^k}$ choices of $L$ that do not form a combinatorial subspace (since these are subsets of $[q]^k$).

We claim that the strict inequality $\sum_{\cc\in L}N_{\cc,D}< q^k$ also holds in case~\ref{case_smallboxes}. This is immediate from Lemma~\ref{lem:ineq_composite} if $L$ is not a combinatorial subspace as clearly $N_{\cc,D}\le N_{\cc,L}$.
If $L$ is a combinatorial subspace, then Lemma~\ref{lem:ineq_composite} only tells us that $\sum_{\cc\in L}N_{\cc,L} = q^k$. However, adding a column corresponding to the small box $V_{\cc^s}$ will necessarily decrease $N_{\cc,D}$ for some $\cc\in L$, i.e.\ $N_{\cc,D}<N_{\cc,L}$. Indeed, when $L$ is of the form \eqref{e:comb} there must be some $j_1$, $j_2$ with $i_{j_1}=i_{j_2}$ but $\phi_{j_1}^{-1}(c^s_{j_1})\ne\phi_{j_2}^{-1}(c^s_{j_2})$. But then adding a column corresponding to $\cc^s$ to $M^{(\cc,L)}$ for any $\cc$ with $c_{j_1}=c^s_{j_1}$ and hence $c_{j_2}\ne c^s_{j_2}$ will ensure that $|\Mcol{\cc,D}|>|\Mcol{\cc,L}|$, and hence $N_{\cc,D}<N_{\cc,L}$. Thus, we again get $\sum_{\cc\in D}N_{\cc,D}<\sum_{\cc\in L}N_{\cc,L} \leq q^k$.

Now using an analogous calculation to case \ref{case_noplane}, the configurations of type \ref{case_smallboxes} with fixed $L$ and $\cc^s$ contribute at most
\begin{align}\label{eq:nochcontr2}
 \sum_{(|V_\cc|:\cc\in L)}\frac{n!}{\prod_{\cc \in L} |V_\cc|!}\cdot q^{-kn}q^{q^{2k}}\prod_{\cc\in L} N_{\cc,D}^{|V_\cc|}
 &= q^{q^{2k}}q^{-kn}n\Big(\sum_{c\in L}N_{\cc,D}\Big)^{n-1} \nonumber\\
 &\le q^{q^{2k}}q^{-kn}n(q^k-1)^{n-1} \nonumber\\
 &= q^{q^{2k}}q^{-k}n(1-q^{-k})^{n-1}.
\end{align}
The second result follows as there are at most $q^k$ choices of $\cc^s$ and $2^{q^k}$ choices of~$L$.
\end{proof}

Lemma~\ref{lemma:zetachoice}\ref{case_noplane} directly yields that $k$-tuples of partitions for which $L$ is not a combinatorial subspace and there are no small boxes, contribute $o(1)$ to \eqref{eq:mainexp_mod}. Indeed, by Lemma~\ref{lemma:removebad} we need only consider configurations without bad vertices and by Lemma~\ref{lemma:nolargeconflict} these do not admit any conflicts. These together provide the hypothesis of Lemma~\ref{lemma:zetachoice}\ref{case_noplane}.

The remaining $k$-tuples of partitions are those for which there is at least one small box.  Again by Lemma~\ref{lemma:removebad} and Lemma~\ref{lemma:nolargeconflict}, we may assume that in all such configurations all vertices $v \in V_{\cc'}$ are good and have the same value of $\zeta_{v,I_{\cc,{\cc'}}}$ for all large boxes $V_{\cc}$.

Fix a small box $V_{\cc^s} \ne\emptyset$ and an arbitrary vertex $v^\star\in V_{\cc^s}$. Remove all vertices in all small boxes except $v^\star$ and all vertices from large boxes that are in conflict with~$v^\star$. Then replace each removed vertex by a duplicate most common vertex. Note that the resulting configurations do not have any conflicts and hence, by Lemma~\ref{lemma:zetachoice}\ref{case_smallboxes}, contribute at most $2^{q^k}q^{q^{2k}}n(1-q^{-k})^{n-1}$ to \eqref{eq:mainexp_mod}. The total number of vertices replaced by most common vertices is $t\le q^k(2C\log n+2)$ as small boxes contain at most $2C\log n+2$ vertices and $v^\star$ was not in conflict with more than $C \log n$ vertices in any large box. As previously noted, this construction sends at most $\sum_t (nq^{2k})^t$ configurations to a single resulting configuration. Therefore, the total contribution to \eqref{eq:mainexp_mod} from configurations with small boxes is at most
\[
 \sum_{t\le q^k(2C\log n+2)} (nq^{2k})^t\cdot 2^{q^k}q^{q^{2k}}n(1-q^{-k})^{n-1} \leq e^{O((\log n)^2)-q^{-k}n}=o(1).
\]
This completes the proof of Theorem~\ref{thm:main}.

\section{Proof of Theorem~\ref{thm:main2}}\label{sec:proof2}

As the proof of Theorem~\ref{thm:main2} is very similar to that of Theorem~\ref{thm:main}, we will only describe the points at which the calculation deviates. 

Fix non-negative integers $a_0,\dots,a_{q-1}$ that satisfy $\sum_{x=0}^{q-1}a_x=q$. Recall that $T$ is the index set $\{(x,y): x=0,\dots,q-1, y=1,\dots, a_x\}$. We shall say that a vertex partition $A$ into sets $A_{x,y}$ with $(x,y) \in T$ is \emph{good} if 
\[
 d_{G[A_{x,y}]}(v) \equiv x \bmod q
\]
for all $(x,y)\in T$ and $v\in A_{x,y}$. We again begin by determining the probability that each partition in a fixed $k$-tuple $\A = (A^{(1)},A^{(2)},\dots,A^{(k)})$ is good. Assign roots of unity $\zeta_{v,j}$ as before and let $R$ again denote the set of all possible assignments of roots to vertices for each partition. For each vertex $v$ we have
\begin{equation}\label{eq:zetacases2}
 \frac{1}{q}\sum_{\zeta \in \mu_q} \zeta^{d_{G_j}(v)-x_{v,j}} = \begin{cases}
    1 & \text{if } d_{G_j}(v) \equiv x_{v,j} \bmod q \\
    0 & \text{otherwise} \\
  \end{cases}
\end{equation}
where $x_{v,j}$ represents the congruence class corresponding to the part of $A^{(j)}$ to which the vertex $v$ belongs (i.e. $ v \in (A^{(j)})_{x_{v,j},y}$ for some $y$). Letting $Y$ be the event that all $A^{(j)}$ are good, then
\begin{align*}
\mathbb{P}(Y)
&= \E \frac{1}{q^{kn}} \sum_{\zzeta \in R} \prod_{j \in [k]} \prod_{v \in [n]} \zeta_{v,j}^{d_{G_j}(v)-x_{v,j}}\\
&= \frac{1}{q^{kn}} \sum_{\zzeta\in R} \Bigg(\prod_{j \in [k]} \prod_{v \in [n]} \zeta_{v,j}^{-x_{v,j}}\Bigg) \Bigg(\E\prod_{j \in [k]} \prod_{vw \in E(G_j)} \zeta_{v,j} \zeta_{w,j} \Bigg) \\
&= \frac{1}{q^{kn}} \sum_{\zzeta\in R} \Bigg(\prod_{j \in [k]} \prod_{v \in [n]} \zeta_{v,j}^{-x_{v,j}}\Bigg) \Bigg(\E \prod_{vw \in E(G)} \prod_{j \in I_{v,w}} \zeta_{v,j} \zeta_{w,j}\Bigg)\\
&= \frac{1}{q^{kn}} \sum_{\zzeta\in R} \Bigg(\prod_{j \in [k]} \prod_{v \in [n]} \zeta_{v,j}^{-x_{v,j}}\Bigg) \Bigg( \prod_{\{v,w\}} \frac{1}{2}\Big(1+ \prod_{j \in I_{v,w}} \zeta_{v,j} \zeta_{w,j}\Big)\Bigg).
\end{align*}
Hence, we can write 
\begin{equation*}
 \E((X_n)_k)=\frac{1}{q^{kn}} \sum_{\A} \sum_{\zzeta\in R}\Bigg(\prod_{j \in [k]} \prod_{v \in [n]} \zeta_{v,j}^{-x_{v,j}} \Bigg)\Bigg( \prod_{\{v,w\}}\frac{1}{2}\Big(1+\prod_{j\in I_{v,w}}\zeta_{v,j}\zeta_{w,j}\Big)\Bigg)  
\end{equation*}
with the first sum taken over all $k$-tuples of distinct partitions. Following the arguments of the proof of Theorem~\ref{thm:main} we can show that summing over all configurations $(\A,\zzeta)$ with conflicts, or with $L_{\A}\ne [q]^k$,
\[
 \frac{1}{q^{kn}} \sum_{(\A,\zzeta)}\left|\prod_{\{v,w\}}\frac{1}{2}\Big(1+\prod_{j\in I_{v,w}}\zeta_{v,j}\zeta_{w,j}\Big)\right|=o(1).
\]
Furthermore we have $\big|\prod_{j \in [k]} \prod_{v \in [n]} \zeta_{v,j}^{-x_{v,j}}\big| = 1$, giving that the contribution from these configurations is indeed still~$o(1)$.

For configurations without conflicts and with all boxes large, note that $\prod_{\{v,w\}}\frac{1}{2}\Big(1+\prod_{j\in I_{v,w}}\zeta_{v,j}\zeta_{w,j}\Big)$ is always 1 and recall that for a fixed $\A$ the value of $\zeta_{v,j}$ depends only on the part of the $j$th partition to which $v$ belongs. As $\zeta_{v,j}=\zeta_{t,j}\in\{\pm1\}$ when $v\in (A^{(j)})_t$, this allows us to write
\[
 \prod_{j \in [k]} \prod_{v \in [n]} \zeta_{v,j}^{-x_{v,j}}
 =\prod_{t=(x,y)\in T } \prod_{j \in [k]}\zeta_{t,j}^{-x|(A^{(j)})_{t}|}.
\]
If $q$ is odd then all $\zeta_{t,j}=1$ for all $t$ and $j$, and this factor
is~1. If $q$ is even then we must sum over choices of $\zeta_{t,j}\in\{\pm1\}$. We find that there is no contribution from $k$-tuples of partitions $\A$ such that there is some $t=(x,y)\in T$ and $j\in[k]$ with $x|(A^{(j)})_t|$ odd, as then 
\[
\sum_{(\zeta_{t,j})}\Bigg( \prod_{t=(x,y)\in T } \prod_{j \in [k]} \zeta_{t,j}^{-x|(A^{(j)})_{t}|} \Bigg)=0.
\]
Hence only tuples $\A$ where $x|(A^{(j)})_t|$ is even for all $t$ and $j$ contribute, in which case we get a factor of $2^{kq}$ when we sum over the choices of the $\zeta_{t,j}$. Thus for $q$ even and recalling that $c=\sum a_{2x}$ is the number of parts with even degree conditions, we shall insist that the $q-c$ parts $A_{x,y}$ where $x$ is odd are of even size in all partitions.

To count the number of such partitions, it is convenient to work through \emph{ordered\/} partitions. For this, let us view the congruence conditions as a list consisting of $a_x$ entries equal to $x$ for each $x = 0,\ldots q-1$, and ordered so that all $q-c$ odd values appear before the $c$ even ones. 
So, let $P$ be the number of ordered partitions of $[n]$ with the additional condition that when $q$ is even the first $q-c$ parts have even size. 
Allowing for some of the parts to be empty, we note that $P$ is also the number of functions $[n]\to[q]$ where the preimages of $1,\dots,q-c$ are of
even size. The number of functions mapping $n_i$ elements to $i$ for $i=1,\dots,q$ is precisely the coefficient of $z_1^{n_1}z_2^{n_2}\dots$ in the multinomial expansion of $(z_1+z_2+\dots+z_q)^n$.
Thus, the number for which $n_1,\dots,n_{q-c}$ are even can be obtained by averaging this expression over all choices of $z_i=\pm1$ when $i\le q-c$ while fixing $z_i=1$ for $i>q-c$, i.e.\
\[
 \frac{1}{2^{q-c}}\sum_{\eps_i\in\{\pm1\}}(\eps_1+\dots+\eps_{q-c}+1+\dots+1)^n
 =\frac{1}{2^{q-c}}\sum_{i=0}^{q-c}\binom{q-c}{i}(q-2i)^n.
\]
When $c>0$, this is $2^{-(q-c)}(q^n+O((q-2)^n))$ as $n\to\infty$. However, when $c=0$
this becomes $2^{-q}(q^n+(-q)^n+O((q-2)^n))$. Thus the number of ordered partitions with the first $q-c$
parts of even size is $(1+o(1))2^{-(q-c)}q^n$ if $c>0$, and $(1+o(1))2^{-(q-1)}q^n$ if $c=0$ and $n$ is even
(and in fact precisely zero if $c=0$ and $n$ is odd). To summarise, we have
\[
 P=\begin{cases}
  q^n,&\text{$q$ odd,}\\
  (1+o(1))2^c\cdot q^n/2^q,&\text{$q$ even, $c>0$,}\\
  (1+o(1))2\cdot q^n/2^q,&\text{$q$ even, $c=0$, $n$ even.}
 \end{cases}
\]
It remains to adjust our counts for unordered partitions. Taking $k$-tuples of ordered partitions, it is possible to have a degenerate case in which two partitions $A^{(i)}$ and $A^{(j)}$ are the same unordered partition. Now either $x_{1} = x_{2}$ whenever $(A^{(i)})_{x_1,y_1}=(A^{(j)})_{x_2,y_2}$ meaning that corresponding parts have the same congruence condition, or else they are incompatible. For the former case, by ordering the elements of $\{A_{x,1}, \ldots, A_{x, a_x}\}$ for each $x\in \{0,\ldots, q-1\}$, each unordered good partition without empty parts gives rise to $F:=\prod_{x=0}^{q-1} a_x!$ ordered partitions with identical degree conditions. All of these count as a single partition when calculating the random variable~$X_n$.

The latter case occurs when there are $(A^{(i)})_{x_1,y_1}=(A^{(j)})_{x_2,y_2}$ for which $x_{1} \neq x_{2}$. However, it is then impossible for the vertices of $G[(A^{(i)})_{x_1,y_1}]=G[(A^{(j)})_{x_2,y_2}]$ to satisfy both degree conditions simultaneously so such $k$-tuples do not contribute to $\E((X_n)_k)$. 
Thus, again noting that only a $o(1)$ proportion of the partitions counted by $P$ have empty parts, there are $((1+o(1))P/F)_k$ tuples $\A$ of partitions where any part with an odd degree condition is even when $q$ is even. As $o(q^{kn})$ of these tuples have small boxes, we find that there are indeed $((1+o(1))P/F)_k$ tuples contributing $1$ each when $q$ is odd and $2^{qk}$ each when $q$ is even. Thus for $\E((X_n)_k)$ we get $q^{-kn}(P/F)_k(1+o(1))$ when $q$ is odd, and $q^{-kn}(P/F)_k2^{qk}(1+o(1))$ when $q$ is even.
The result now follows by application of Theorem~\ref{thm:poisson}. 

\section{Proofs of algebraic lemmas}\label{sec:LA}

In this section we prove the deferred algebraic lemmas used in the proofs of our main theorems.

The following general statement implies Lemma~\ref{lem:nc}. The proof is standard, but we include it for completeness.
\begin{lemma}\label{lem:nc_alg_sec}
Let $M$ be a $k\times \ell$ matrix with entries in $\Z/q\Z$ and\/ $\boldsymbol{b} \in (\Z/q\Z)^\ell$. If there is at least one solution $\boldsymbol{a} \in (\Z/q\Z)^k$ to the congruence
\begin{align}\label{e:congr_alg_sec}
 \boldsymbol{a}M\equiv \boldsymbol{b}\bmod q
\end{align}
then the total number of solutions is given by 
\[
 N = \left|\faktor{(\Z/q\Z)^k}{\Mcol{}}\right|= \frac{q^k}{|\Mcol{}|}.
\]
\end{lemma}
\begin{proof}
Define a group homomorphism 
\begin{align*}
M_{\rm Grp}\colon (\Z/q\Z)^k &\to (\Z/q\Z)^{\ell}\\
\boldsymbol{a}&\mapsto \boldsymbol{a}M
\end{align*}
where elements $\boldsymbol{a}$ are viewed as row vectors. The number of solutions to \eqref{e:congr_alg_sec} is the size of the preimage of~$\boldsymbol{b}$, which is either empty or a coset of the kernel. That is, if there exists a solution, the number of solutions is $|\ker M_{\rm Grp}|$.
Consider the natural bijection $\phi\colon (\Z/q\Z)^k \to \Hom((\Z/q\Z)^k, \Z/q\Z)$ given by sending $\boldsymbol{a} \mapsto (\boldsymbol{v}\mapsto \boldsymbol{a}\boldsymbol{v}^T)$. We have $\boldsymbol{a}M=\boldsymbol{0}$ if and only if $\boldsymbol{a}M\boldsymbol{v}^T=\boldsymbol{0}$ for all $\boldsymbol{v}\in (\Z/q\Z)^{\ell}$ viewed as row vectors, or equivalently $\boldsymbol{ab} = 0$ for all $\boldsymbol{b}\in \Mcol{} = \{M\boldsymbol{v}^T:\boldsymbol{v}\in(\Z/q\Z)^{\ell}\}$. Hence $\phi$ induces a bijection between elements of $\ker M_{\rm Grp}$ and morphisms in $\Hom((\Z/q\Z)^k, \Z/q\Z)$ that annihilate $\Mcol{}$. Such morphisms form a subgroup that is naturally isomorphic to $\Hom((\Z/q\Z)^k/\Mcol{},\Z/q\Z)$, which by duality for finite abelian groups is isomorphic to $(\Z/q\Z)^k/\Mcol{}$.
\end{proof}

For Lemma~\ref{lem:ineq_composite}, we recall the surrounding setup from Section~\ref{sec:mainproof}. Since the matrix $M^{(c,B)}$ only appears with $B=L$ in this statement, let us slightly simplify our notation here:
given any subset $L \subseteq [q]^k$, associate each $\cc\in L$ with the matrix
\begin{equation}\label{eq:matrixdef}
M^{(\cc)} =\big(M^{(\cc)}_{j,\cc'}\big)_{j\in[k],\cc'\in L}
\end{equation}
and let 
\[
N_\cc = \left|\faktor{(\Z/q\Z)^k}{\Mcol{\cc}}\right| = \frac{q^k}{|\Mcol{\cc}|}.
\]
Also recall that $q\geq3.$

\begin{proof}[Proof of Lemma~\ref{lem:ineq_composite}]
We use induction on $k$. The case $k=1$ is clear as the $\cc$ column of $M^{(\cc)}$ is $(1)$ and so generates the full group $\Z/q\Z$. Thus
\[
\sum_{c\in L} N_c = \sum_{c\in L} 1 \leq q,
\]
with equality if and only if $L = [q]$. 

Now assume that $k>1$. Define the \emph{$k$-compression} $\Gamma_k=\Gamma\colon [q]^k\to [q]^{k-1}$ for $\cc=(c_1,\dots,c_{k})$ by $\Gamma(\cc)=(c_1,\dots,c_{k-1})$,
and let $\Gamma(L)=\{\Gamma(\cc)\colon \cc \in L\}$. For $\Gamma(\cc)\in [q]^{k-1}$ we define $M^{(\Gamma(\cc))}$ and $N_{\Gamma(\cc)}$
as before, but in $k-1$ dimensions and with respect to $\Gamma(L)$. We note that the matrix $M^{(\Gamma(\cc))}$ is obtained from $M^{(\cc)}$ by deleting the $k$th row and then possibly deleting some duplicate columns, so $N_\cc=t\cdot N_{\Gamma(\cc)}$ for some $t\mid q$. Indeed, this $t$ is the smallest positive integer
such that $te_k\in \Mcol{\cc}$ where $e_k=(0,\dots,0,1)^T$. To see this, note that the kernel of the surjective group homomorphism $\Mcol{\cc}\to\Mcol{\Gamma(\cc)}$ which forgets the last coordinate is precisely $\langle te_k\rangle$, so
$N_\cc/N_{\Gamma(\cc)}=q/(|\Mcol{\cc}|/|\Mcol{\Gamma(\cc)}|)=q/|\langle te_k\rangle|=t$.

We have by induction that
\begin{equation}\label{e:qsumind}
 \sum_{\ca\in \Gamma(L)} N_{\ca}\le q^{k-1},
\end{equation}
with equality if and only if $\Gamma(L)$ is a combinatorial subspace.

Fix some $\ca \in \Gamma(L)$ and first suppose that there are $\cc_1,\cc_2\in L$ with $\cc_1\ne \cc_2$ but $\Gamma(\cc_1)=\ca=\Gamma(\cc_2)$.
These conditions imply that $\cc_1$ and $\cc_2$ differ only in the $k$th coordinate,
so subtracting the columns corresponding to $\cc_1$ and $\cc_2$
shows that $e_k$ is in $\Mcol{\cc_1}$ and $\Mcol{\cc_2}$ and hence $N_{\cc_1}=N_{\cc_2}=N_{\ca}$.

If on the other hand $\Gamma^{-1}(\ca)\cap L=\{\cc\}$ for $\ca \in \Gamma(L)$, then $N_{\cc}=tN_{\ca}\le qN_{\ca}$.
Thus we see that
\begin{equation}\label{e:qsumline}
 \sum_{\cc:\Gamma(\cc)=\ca} N_\cc\le q\cdot N_{\ca},
\end{equation}
with equality when $|\{\cc\in L:\Gamma(\cc)=\ca\}|=q$ or when $|\{\cc\in L:\Gamma(\cc)=\ca\}|=1$ and $te_k\notin \Mcol{\cc}$ for all $t<q$ where $\cc$ is the unique preimage of~$\ca$.
Hence we deduce that
\begin{equation}\label{e:qsum1}
 \sum_{\cc\in L} N_\cc\le q\sum_{\ca \in\Gamma(L)} N_{\ca}\le q\cdot q^{k-1}=q^k.
\end{equation}
as required.

Now for equality we must have $|\{\cc\in L:\Gamma(\cc)=\ca\}|\in \{1,q\}$ for every $\ca\in\Gamma(L)$.
However, if there exists any such $\ca$ with $|\{\cc\in L:\Gamma(\cc)=\ca\}|=q$ then
$e_k$ lies in the column span of every $M^{(\cc)}$. Indeed, we need only subtract the columns
corresponding to $\cc'$ and $\cc''$, where $\Gamma(\cc')=\Gamma(\cc'')=\ca$, $c'_k=c_k$ and $c''_k\ne c_k$.
Such $\cc'$ and $\cc''$ exist as all $\cc$ with $\Gamma(\cc)=\ca$ lie in~$L$.
Hence, for equality to hold in \eqref{e:qsum1} we would now need $|\{\cc\in L:\Gamma(\cc)=\ca\}|=q$ for
all $\ca\in\Gamma(L)$.

Thus we are reduced to two cases. The first is that $L=\Gamma(L)\times [q]$. In this case $N_\cc=N_{\Gamma(\cc)}$
for all $\cc\in L$ and equality in \eqref{e:qsum1} occurs if and only if it does in \eqref{e:qsumind}, which in turn happens if and only if  $\Gamma(L)$ is a combinatorial subspace. But $\Gamma(L)$ is a combinatorial subspace if and only if $L=\Gamma(L)\times[q]$ is a combinatorial subspace.

The second case is that $|\{\cc\in L:\Gamma(\cc)=\ca\}|=1$ for all $\ca \in \Gamma(L)$.
Again, for equality we must have $\Gamma(L)$ equal to a combinatorial subspace, and moreover
$e_k$ is not in the subgroup $\Mcol{\cc}$ for any $\cc\in L$. Writing $\Gamma(L)$ as in \eqref{e:comb}, we may assume all the $\phi_j$ are equal to the identity as that just corresponds to permuting the values in $[q]$ in the $j$th coordinate and this does not affect the matrix. We may also suppose that all the $i_j$ are distinct, since if any component of the combinatorial subspace is bound to another, say $x_{i_a}=x_{i_b}$, this would mean that rows $a$ and $b$ are identical in every matrix $M^{(\cc)}$. It is clear that deleting duplicate rows does not affect $|\Mcol{\cc}|$. Hence for the $a$-compression $\Gamma_a$ obtained by removing coordinate~$a$, we have $qN_{\Gamma_a(\cc)}=N_{\cc}$ for all $\cc$ where $N_{\Gamma_a(\cc)}$ is defined with respect to $\Gamma_a(L)$ in $k-1$ dimensions. Now we only have equality if $\Gamma_a(L)$ is a combinatorial subspace. But that implies that $L$ is also a combinatorial subspace as, given a combinatorial subspace representation of $\Gamma_a(L)$, we can bind coordinate $a$ to~$b$.

Thus we are reduced to the case $\Gamma(L)=[q]^{k-1}$. Define $L(i)=\{\Gamma(\cc):\cc\in L,\,c_k=i\}$ to be the layer consisting of elements of $\Gamma(L)=[q]^{k-1}$
whose corresponding element $\cc\in L$ has $k$th coordinate equal to~$i$.
As there are only $q$ possible values of~$i$, there must exist an $i$ for which $|L(i)|\ge q^{k-2}$.

Fix such an $i$ and for $\ca_1=\Gamma(\cc_1),\ca_2=\Gamma(\cc_2)\in L(i)$ let $S_{12}=\{j\in[k-1]:(\ca_1)_j\ne (\ca_2)_j\}$.
We note that $(\id\{j\in S_{12}^c\})_j^T$ is equal to the column corresponding to $\cc_2$ in $M^{(\cc_1)}$.
Also, if we have some $\ca_3=\Gamma(\cc_3)\in L(i')$ for $i'\ne i$ and $S_{13}=\{j:(\ca_1)_j\ne (\ca_3)_j\}$
is equal to $S_{12}$, then subtracting columns $\cc_2$ and $\cc_3$ in $M^{(\cc_1)}$ gives~$e_k$ which is a contradiction.
Hence, every $\ca_3$ with $S_{12}=S_{13}$ must lie in $L(i)$.

Now fix $\ca_1,\ca_2\in L(i)$ so as to maximise $|S_{12}|$.
We claim that $\ca_3\in L(i)$ if and only if $S_{13}\subseteq S_{12}$. First suppose there exists
some $\ca_3\in L(i)$ with $S_{13}$ not a subset of $S_{12}$. 
Then take some $\ca_4 \in \Gamma(L)=[q]^{k-1}$ for which $(\ca_4)_j=(\ca_3)_j$ for $j\notin S_{12}\cap S_{13}$, and in addition so that $(\ca_4)_j$ is distinct from $(\ca_1)_j$ and $(\ca_2)_j$ for all $j\in S_{12}\cap S_{13}$. Such a $\ca_4$ can be found since $q\ge 3$, meaning we have room to pick component values satisfying these conditions. With this choice of $\ca_4$, observe that $S_{14}=S_{13}$ and hence $\ca_4\in L(i)$. However $S_{24}=S_{12}\cup S_{13}$ then has more elements than $S_{12}$, contradicting the choice of $\ca_1$ and $\ca_2$.

Conversely, suppose $S_{13}\subseteq S_{12}$. Pick $\ca_4$ such that $(\ca_4)_j=(\ca_2)_j$
if $j\notin S_{13}$, and $(\ca_4)_j\ne (\ca_1)_j,(\ca_3)_j$ when $j\in S_{13}$.
Then $S_{14}=S_{12}$, so $\ca_4\in L(i)$. Also $S_{34}=S_{14}$, so (basing our arguments at $\ca_4$
and noting that $\ca_1,\ca_4\in L(i)$) we see that $\ca_3\in L(i)$.

Now given that $\ca_3\in L(i)$ if and only if $S_{13}\subseteq S_{12}$, as $|L(i)|\ge q^{k-2}$
we must have $|S_{12}|=k-2$ or $k-1$. If $|S_{12}|=k-1$, then $L(i)=\Gamma(L)$ and so
every other $L(j)$ is empty. In this case for any $\cc\in L$ we can find $\cc' \in L$ that differs from $\cc$ in all but the $k$th coordinate. This means that $e_k$ is a column of every $M^{(\cc)}$ and we do not have equality.

The other possibility is that $|S_{12}|=k-2$, in which case $S_{12}$ omits exactly one coordinate. Without loss of generality, say it is coordinate~1.
But then $L(i)=\{\ca\in [q]^{k-1}:(\ca)_1=x_i\}$ for some $x_i\in [q]$.
Now the remaining $L(j)$ have $(q-1)q^{k-2}$ elements in total, so there is some 
$j\neq i$
with $|L(j)|\ge q^{k-2}$. By the above argument this is again a co-dimension 1 subspace
given by an equation of the form $(\ca)_\ell=x_j$. But $L(j)\cap L(i)=\emptyset$, so $\ell=1$
and $x_j\ne x_i$.

Continuing in this manner, we see that each $L(j)=\{\ca:(\ca)_1=x_j\}$
for distinct elements $x_1,\dots,x_q\in [q]$. Writing $\phi_1(j)=x_j$, $\phi_i(j)=j$ for all $i>1$,
$i_j=j$ for $j<k$ and $i_k=1$, and $r=k-1$, we obtain a form satisfying~\eqref{e:comb} so $L$ is in fact a combinatorial subspace. Moreover, rows $1$ and $k$
of each $M^{(\cc)}$ are identical so we have $N_\cc\ge q$ for all $\cc\in L$. Hence equality in \eqref{e:qsum} must hold.
\end{proof}

\section{When $q=2$}\label{sec:final}

Observe that we have crucially used the assumption that $q>2$ in both the proof of Lemma~\ref{lem:ineq_composite} and the main body of the proof of Theorem~\ref{thm:main}. In fact, one can show directly that it is not possible to have a Poisson distribution when $q=2$. In this section, we determine the distribution of the number of partitions into two parts both inducing even graphs, both inducing odd graphs, and inducing one odd and one even graph.

It is convenient to work in the setting of uniform random symmetric matrices over $\F_2=\Z/2\Z$, that is, symmetric matrices with entries on the diagonal and upper triangle chosen independently randomly to be 0 or 1 each with probability $1/2$. Henceforth, we will just call these \emph{random symmetric matrices}. Let $\ones = \ones_n$ denote the $n\times 1$ vector of all 1s, $\zeroes = \zeroes_n$ denote the $n\times 1$ vector of all 0s, and $I=I_n$ denote the $n\times n$ identity matrix. We shall call a vector $\vv$ even if $\ones^T \vv=0$, and odd otherwise. Then say that a matrix $A$ is even if $\ones^TA=\zeroes^T$, and odd if $\ones^TA=\ones^T$. In other words, a matrix is even if all column sums are even, and similarly for odd.

To bring this back to our partitions, let $G$ be a graph on $n$ vertices with degree sequence $d_1,\dots,d_n$, and let $A$ be its adjacency matrix. Then $A$ is a random symmetric matrix with a fixed diagonal of 0s for $G=G_{n,1/2}$. Let $D$ be the $n\times n$ diagonal matrix with entries on the diagonal being $d_i \bmod 2$. Any ordered partition of the vertex set into $q=2$ parts can be represented by a 0-1 column vector of length~$n$, where all of the positions with 0 entries form one part, whilst the 1 entries form the other. 

For each degree parity condition on the two parts (even/even, even/odd, odd/odd), we shall set up a nonhomogeneous linear system $M \vv= \bb$ where $M$ is a random symmetric matrix over $\F_2$ that is either even or odd, and the solutions are precisely good (ordered) partitions. This immediately implies that the number of solutions, if there are any, must be a power of 2 and hence already precludes the possibility of a Poisson distribution. To obtain the distribution, it suffices to know the probability that the system is inconsistent (this can only occur in the odd/odd case), together with the rank distribution of random symmetric matrices. The latter can be deduced from the following counts.

\begin{theorem}[MacWilliams~\cite{Mac69}]\label{thm:rankcounts}
For $0\leq r\leq n$, the number of symmetric $n\times n$ matrices over $\F_2$ with rank $r$ is 
\[
\prod_{i=1}^{\lfloor r/2 \rfloor}  \frac{2^{2i}}{2^{2i}-1} \cdot \prod_{i=0}^{r-1} (2^{n-i}-1).
\]
\end{theorem}

\begin{corollary}\label{cor:evenrank}
The probability of a uniformly random even symmetric $n\times n$ matrix having rank $r$ is
\[
2^{-\binom{n}{2}} \cdot \prod_{i=1}^{\lfloor r/2 \rfloor}  \frac{2^{2i}}{2^{2i}-1} \cdot \prod_{i=0}^{r-1} (2^{n-i-1}-1)
\]
if\/ $0\le r\leq n-1$, and\/ $0$ otherwise.
\end{corollary}
\begin{proof}
In an even symmetric matrix $M$, the entries in the last row and column are completely determined by the minor $(M)_{n,n}$. Explicitly, we can uniquely recover $M$ by adding an $n$th row given by the sum of the existing rows in the $(n-1)\times(n-1)$ matrix, and then the $n$th column given by the sum of columns in the $n\times(n-1)$ matrix. As this preserves rank, it follows that the rank distribution of a random even symmetric $n\times n$ matrix is the same as that of a random symmetric $(n-1)\times(n-1)$ matrix with no parity conditions. This shows $M$ cannot have full rank, and that the probability of having rank $r\leq n-1$ is equal to the probability that a (fully) random symmetric $(n-1)\times (n-1)$ has rank~$r$. The required probability now follows directly from Theorem~\ref{thm:rankcounts}, noting that there are $2^{\binom{n}{2}}$ many $(n-1)\times (n-1)$ symmetric matrices.
\end{proof}

We now address each of the three possible parity conditions in turn.
\begin{theorem}
 Let\/ $G=G_{n,1/2}$, and let\/ $X_n=X_n^{(2,0)}$ be the number of partitions of\/ $V(G)$
 into two sets both inducing even subgraphs. Then
 \[
 \Prb(X_n=2^k) = 2^{-\binom{n}{2}} \cdot \prod_{i=1}^{\lfloor (n-1-k)/2 \rfloor}  (1-2^{-2i})^{-1}\cdot \prod_{i=k+1}^{n-1} (2^i-1)
 \]
 for $0\le k\leq n-1$, and if\/ $x \neq 2^k$ for any such $k$ then $\Prb(X_n=x)=0$.
\end{theorem}
\begin{proof}
We use the notation defined at the start of the section. Suppose we have an ordered partition of $V(G)$ into two parts, represented by a 0-1 column vector $\vv$. The entries of $A\vv$ represent degrees into the 1-part of the partition, and entries of $D(\ones-\vv)$ are total degrees (i.e. degrees in $G$) for vertices of the 0-part and are 0 for vertices of the 1-part. Hence if $\vv$ corresponds to a vertex partition into two parts such that both induce even subgraphs, then it is a solution to 
\[
 A\vv=D(\ones-\vv)\iff (A+D)\vv=D\ones,
\]
and conversely. The solutions $\vv$ to this linear equation form an affine subspace of $\F_2$ and Gallai's theorem ensures that there is at least one solution $\vv$ to this linear equation. Hence, the probability that there are $2^k$ \emph{unordered} partitions is equal to the probability that $A+D$ has rank $r=n-1-k$. Noting that $A+D$ is a uniformly random even symmetric matrix, the even condition being guaranteed by definition of $D$, the result now follows from Corollary~\ref{cor:evenrank}. 
\end{proof}

\begin{theorem}
 Let\/ $G=G_{n,1/2}$, and let\/ $Y_n=X_n^{(1,1)}$ be the number of partitions of\/ $V(G)$ into two sets,
 one inducing an even subgraph and the other inducing an odd subgraph. If\/ $n$ is odd, then
 \[
  \Prb(Y_n=2^k) = 2^{-\binom{n}{2}} \cdot \prod_{i=1}^{\lfloor (n-1-k)/2 \rfloor}(1-2^{-2i})^{-1}\cdot \prod_{i=k+1}^{n-1} (2^i-1)
 \]
 for $0\leq k \leq n-1$, and\/ $\Prb(Y_n=x)=0$ when $x \neq 2^k$ for any such~$k$.
 If\/ $n$ is even, then $\Prb(Y_n=x) = \Prb(Y_{n-1}=x)$.
\end{theorem}
\begin{proof}
It suffices to count ordered partitions for which the 1-part is odd and the 0-part is even since every unordered partition satisfying the degree conditions has a unique ordering for which this is true. The partitions we wish to count are then precisely the solutions to
\[
 A\vv=D(\ones-\vv)+\vv \iff (A+D+I)\vv= D\ones.
\]
By the corollary to Gallai's theorem, this system has at least one solution. Thus, to determine the distribution of $Y_n$ it suffices to know the rank distribution of $O=A+D+I$. 

Let $n$ be odd. Since $A+D$ is a uniform random even symmetric matrix, both $A+D+I$ and $A+D+\ones\ones^T$ are uniform random odd symmetric matrices and hence have the same rank distribution as each other. We will work with the latter. Note that $\rank(A+D+\ones\ones^T) \leq \rank(A+D)+\rank(\ones\ones^T) = \rank(A+D)+1$.
The column space of $(A+D+\ones\ones^T)$ contains $(A+D+\ones\ones^T)\ones = \ones$ and thus also contains the column space of $A+D$ as, for any vector $\vv$, we have $(A+D+\ones\ones^T)\vv \in \{(A+D)\vv, (A+D)\vv + \ones\}$. The odd vector $\ones$ is not in the column space of $A+D$, giving  $\rank(A+D)<\rank(A+D+\ones\ones^T)$ and hence $\rank(A+D+\ones\ones^T)= \rank(A+D)+1$.

Then for $1\leq r \leq n$ we have
\begin{align*}
\Prb(Y_n=2^{n-r})&=\Prb(\rank(A+D+I)=r)\\
&=\Prb(\rank(A+D+\ones\ones^T)=r)\\
&=\Prb(\rank(A+D)=r-1)
\end{align*}
which is given by Corollary~\ref{cor:evenrank}. 

For $n$ even, the odd random symmetric matrices $O=A+D+I$ and $O+\ones\ones^T$ have the same rank. To see this, it is enough to note that 
\[
\{O\vv, O(\vv+\ones)\} =\{(O+\ones\ones^T)\vv, (O+\ones\ones^T)(\vv+\ones)\}
\]
for all~$\vv$, which is easily verified by considering the parity of the number of 1s in $\vv$. Thus, we may assume that $O$ is an odd random symmetric matrix with a 1 in the top-left entry as this still has the same rank distribution.

Note that since $O$ is odd, it has $\ones^T$ in its row space. We now add $\ones^T$ to all rows of $O$ which have a $1$ in the first entry, except the first row. The resulting matrix $O'$ still has both row and column sums odd as we added an even vector to an even number of rows. In particular, $\ones^T$ is still in the row space and we thus did not change the row space, i.e.\ $\rank(O')=\rank(O)$. 
In the same way, we can add $\ones$ to all columns of $O'$ which have a $1$ in the first entry, except the first column to obtain an odd symmetric matrix $O''$ which has the same column space as~$O'$. In particular, $\rank(O'')=\rank(O')=\rank(O)$ and the first row and column of $O''$ is $e_1$. The remaining $(n-1)\times(n-1)$ matrix $O_{<11>}$ is an odd (uniform) random symmetric matrix since, given the first row of $O$, we can reconstruct $O$ from $O_{<11>}$.
As $\rank(O_{<11>})=\rank(O)-1$, the distribution of nullities, and hence number of solutions $\vv$, is the same for $n$ as for $n-1$ when $n$ is even. 

\end{proof}

\begin{theorem}\label{t:oo}
 Let\/ $G=G_{n,1/2}$ with $n$ even, and let\/ $Z_n=X_n^{(0,2)}$ be the number of partitions of\/ $V(G)$
 into two sets both inducing odd subgraphs. Then
 \[
 \Prb(Z_n=2^k) = \frac{2^{n-k-1}-1}{2^{\binom{n}{2}}(2^{n-1}-1)} \cdot \prod_{i=1}^{\lfloor (n-1-k)/2 \rfloor}  (1-2^{-2i})^{-1}\cdot \prod_{i=k+1}^{n-1} (2^i-1)
 \]
 for $0\le k\leq n-1$. We have $\Prb(Z_n=x)=0$ if\/ $x \neq 0$ and $x \neq 2^k$ for any $0\le k\leq n-1$, while $\Prb(Z_n=0) \neq 0$ can be obtained from the probability of the complement.
\end{theorem}
\begin{proof}
Fix an even $n$. Ordered partitions for which both parts induce odd subgraphs are given by the solutions of
\[
 A\vv=D(\ones-\vv)+\ones \iff (A+D)\vv= D\ones+\ones,
\]
where $A+D$ is a random even symmetric matrix. Unlike the other degree conditions, it is possible for this system to be inconsistent, meaning $G$ has no such partitions. We can write $\Prb(Z_n=2^k)$ as the probability that there is at least one solution and the rank of $A+D$ is $r=n-k-1$. The probability for the latter is given by Corollary~\ref{cor:evenrank}, so it remains to show that the probability that there is at least one solution conditioned on the rank of $A+D$ being $r$ is $(2^r-1)/(2^{n-1}-1)$.

Let $E=A+D$. By Gallai's theorem we already know that $D\ones$ is in the column space of $E$, so we equivalently determine the probability that $\ones$ is also in the column space of $E$ which is a random even symmetric matrix with given rank $r$. Let $H$ be the set of (not necessarily symmetric) odd invertible matrices $M$ (note that being odd here only requires the condition
$\ones^TM=\ones^T$ on the column sums, the row sums can have either parity.)

We claim that $H$ is a group. It is clear that $H$ contains the identity and is closed under multiplication. Since $H$ is finite and its elements are invertible, these must all have finite order, so it follows that $H$ is also closed under inverses. Now let $\mathcal{E}$ be the set of all even symmetric matrices. Then $H$ acts on $\mathcal{E}$ by the group action which sends $E\mapsto MEM^T$, noting that $\ones^TMEM^T=\ones^TEM^T=0$
so $MEM^T$ is even. 

Given the basis $\{\ee_1+\ee_i\colon i>1\}$ for the subspace of even vectors, and any set of $n-1$ linearly independent even vectors $\vv_2,\dots,\vv_n$, there exists $M\in H$ such that $M(\ee_1+\ee_i)=\vv_i$ for all $i=2,\ldots n$. Indeed, take $M$ to be the matrix with first column $\ee_1$ and $i$th column $\ee_1+\vv_i$ for $i>1$. As $H$ is a group it follows that we can map any $r$-dimensional subspace of even vectors to any other $r$-dimensional subspace of even vectors by a suitable choice of~$M$. Thus, for any fixed $E$
of rank~$r$, the column spaces of the matrices $MEM^T$ hit every $r$-dimensional subspace of even vectors with the same frequency. Thus, the proportion of times $\ones$ lies in the column space is just the proportion of $r$-dimensional subspaces of even vectors that contain $\ones$, which by symmetry is $(2^r-1)/(2^{n-1}-1)$.
Picking $M$ and $E$ uniformly independently randomly with $\rank(E)=r$ means that $MEM^T$ is a uniform random even symmetric matrix
with rank~$r$, so the probability that $\ones$ is in the column space of $E$ conditioned on the rank is as claimed.
\end{proof}
Finally, we deduce the asymptotic results stated in the introduction.

\begin{proof}[Proof of Theorem~\ref{thm:q=2asympt}]
For $k\in \N\cup \{0\}$ and $n$ even we have
\begin{align*}
\Prb(Z_n=2^k) &= \frac{2^{n-k-1}-1}{2^{\binom{n}{2}}(2^{n-1}-1)} \cdot \prod_{i=1}^{\lfloor (n-1-k)/2 \rfloor}  (1-2^{-2i})^{-1}\cdot \prod_{i=k+1}^{n-1} (2^i-1)
\\&= \frac{2^{n-k-1}-1}{2^{n-1}-1}\prod_{i=1}^{\lfloor (n-1-k)/2 \rfloor}  (1-2^{-2i})^{-1} \cdot \prod_{i=1}^{n-1} (1-2^{-i}) \prod_{i=1}^{k} (2^i-1)^{-1} 
\\& \xrightarrow{} 2^{-k} \prod_{i=0}^{\infty}(1-2^{-2i-1}) \prod_{i=1}^{k} (2^i-1)^{-1}= \Prb(Z=2^k)\quad\text{as $n\to\infty$}.
\end{align*}
Analogously, for $X_n=X_n^{(2,0)}, Y_n=X_n^{(1,1)}$ and any $k\in \N$ we find that $$\lim_{n\to \infty} \Prb(X_n=2^k)=\lim_{n\to \infty} \Prb(Y_n=2^k) =c \prod_{i=1}^{k} (2^i-1)^{-1},$$
where $c= \prod_{i=0}^{\infty}(1-2^{-2i-1})$ as defined in the statement.
It remains to argue that $\Prb(Z_n=0)\xrightarrow{}\frac{1}{3}$. Note that $\Prb(Z_n=2^k)\leq 2^{-k+1}$ for all $k,n \in \N$, so $\Prb(Z_n=0)$ converges to $1-\sum_{k=0}^{\infty}\Prb(Z=2^k)$. Let
\[
f(x)=c\sum_{k=0}^\infty x^k\prod_{i=1}^k(2^i-1)^{-1}
\]
 be the generating function of the limiting distribution of $\tilde{X_n}$ where $2^{\tilde{X_n}}$ is the number $X_n$ of even/even partitions. Then $f(1)=1$ and 
 \begin{align*}
 f(2x)&=c\sum_{k=0}^\infty x^k (1+2^k-1)\prod_{i=1}^k(2^i-1)^{-1}\\
 &=c\sum_{k=0}^\infty x^k \Big(\prod_{i=1}^k(2^i-1)^{-1} + \prod_{i=1}^{k-1}(2^i-1)^{-1}\Big) =(1+x)f(x).
 \end{align*}
Furthermore, 
\[
g(x):=f(x/2)= c\sum_{k=0}^\infty 2^{-k} x^k\prod_{i=1}^k(2^i-1)^{-1}=\sum_{k=0}^{\infty}x^k\Prb(Z=2^k),
\]
so $1-\lim_{n \to \infty} \Prb(Z_n=0)=g(1)=f(1/2)$. Using $f(1)=1$ and $f(2x)=(1+x)f(x)$ we get at $x=1/2$ that $1=(3/2)f(1/2)$, so $f(1/2)=2/3$, hence $\lim_{n \to \infty} \Prb(Z_n=0)=1/3$.
\end{proof}

\bigskip
\noindent\textbf{Acknowledgments.}
We would like to thank the anonymous referees for their helpful suggestions.

\bibliographystyle{amsplain}
\bibliography{partitioning}
\end{document}